\documentclass{amsart}
\usepackage{amsmath, amsthm, amssymb, amsfonts, amscd, graphicx, verbatim, mathdots}

\newtheorem{theorem}{Theorem}[section]

\newtheorem{proposition}[theorem]{Proposition}

\newtheorem{lemma}[theorem]{Lemma}
\newtheorem{example}[theorem]{Example}
\newtheorem{conjecture}[theorem]{Conjecture}
\newtheorem{remark}[theorem]{Remark}

\newcommand{\cH}{\mathcal{H}}

\newcommand{\Ker}{\text{Ker }}

\newcommand{\mat}[1]{\begin{bmatrix} #1 \end{bmatrix}}

\title[Indefinite determinantal Representations]{Indefinite determinantal representations versus nonsingularities on the noncommutative $d$-torus}

\author[G.J. Groenewald]{Gilbert J. Groenewald}
\address{G.J. Groenewald, School of Mathematical and Statistical Sciences,
North-West University,
Research Focus: Pure and Applied Analytics,
Private Bag X6001,
Potchefstroom 2520,
South Africa}
\email{Gilbert.Groenewald@nwu.ac.za}

\author[S. ter Horst]{Sanne ter Horst}
\address{S. ter Horst, School of Mathematical and Statistical Sciences,
North-West University,
Research Focus: Pure and Applied Analytics,
Private~Bag X6001,
Potchefstroom 2520,
South Africa
and
DSI-NRF Centre of Excellence in Mathematical and Statistical Sciences (CoE-MaSS),
Johannesburg,
South Africa}
\email{Sanne.TerHorst@nwu.ac.za}

\author[H.J. Woerdeman]{Hugo J. Woerdeman}
\address{H.J. Woerdeman, Department of Mathematics, Drexel University, Philadelphia, PA 19104, USA}
\email{hugo@math.drexel.edu}

\thanks{This work is based on the research supported in part by the National Research Foundation of South Africa (Grant Numbers 118513 and 127364). In addition, HW is partially supported by NSF grant DMS-2000037.}


\begin{document}

\begin{abstract}
We show that for a multivariable polynomial $p(z)=p(z_1, \ldots , z_d)$ with a determinantal representation 
$$ p(z) = p(0) \det (I_n- K (\oplus_{j=1}^d z_j I_{n_j}))$$ the matrix $K$ is structurally similar to a strictly $J$-contractive matrix for some diagonal signature matrix $J$ if and only if the extension of $p(z)$ to a polynomial in $d$-tuples of matrices of arbitrary size given by 
\[
p(U_1, \ldots , U_d) = p(0,\ldots,0) \det (I_n\otimes I_m- (K\otimes I_m) (\oplus_{j=1}^d I_{n_j}\otimes U_j)),
\]
where $U_1,\ldots , U_d \in {\mathbb C}^{m \times m}$, $m\in {\mathbb N}$, does not have roots on the noncommutative $d$-torus consisting of $d$-tuples $(U_1, \ldots , U_d)$ of unitary matrices of arbitrary size.
\end{abstract}

\maketitle

\section{Introduction}\label{S:Intro}

The problem of writing polynomials as determinants of linear matrix pencils has attracted a lot of attention in the past few decades; see, e.g., \cite{Kummert, GKW2012, GKVW2,GKVW3, GKVW2013, GKVW1, Woerdeman2013, GIK, Knese2, Knese3}. In this context, for a multivariable polynomial $p(z)=p(z_1, \ldots , z_d)$ in commutative variables $z=(z_1, \ldots , z_d)$, with $p(0)\neq 0$, one seeks for a {\em determinantal representation} of the form 
\begin{equation}\label{DetRep}
p(z) = p(0) \det (I_n- K (\oplus_{j=1}^d z_j I_{n_j}))
\end{equation}
for an $n \times n$ matrix $K$ with $n=n_1 + \cdots + n_d$, and one studies how properties of the matrix $K$ translate to properties of the polynomial $p(z)$, and vice-versa. For instance, it is easy to see that when $K$ is a contraction (resp.\ strict contraction), then $p(z)$ has no roots on the open polydisk $\mathbb{D}^d$ (resp.\ closed polydisk $\overline{\mathbb{D}}^d$). Conversely, for $d=2$, it was shown in  \cite{Kummert} that a nonzero bivariate polynomial $p(z_1,z_2)$ with no roots on $\mathbb{D}^2$ (resp.\ $\overline{\mathbb{D}}^2$) always has a determinantal representation with $K$ a contraction (resp.\ strict contraction). The proof was streamlined in \cite{GKVW2013}. 

In the present paper we focus on determinantal representations for which the matrix $K$ has an (in)definite norm constraint for which $p(z)$ cannot have roots on the $d$-torus $\mathbb{T}^d$. Recall that a matrix $J$ satisfying $J^*=J=J^{-1} $ is called a {\em signature matrix}. A matrix $K$ is called a {\em strict $J$-contraction} if $J-KJK^* >0$. If $K$ in \eqref{DetRep} is a strict $J$-contraction for a diagonal signature matrix $J$, then $p(z)$ has no roots on $\mathbb{T}^d$. In \cite{JW} it was conjectured that for $d=2$, the converse holds as well. 

\begin{conjecture}\label{JW}\cite[Conjecture 3.1]{JW} Let $p(z_1,z_2)$ be a nonzero polynomial of degree $(n_1,n_2)$. Then $p(z_1,z_2)$ is without roots in $\{ (0,0) \} \cup {\mathbb T}^2$ if and only if there is a diagonal signature matrix $J$ and a strict $J$-contraction $K$ so that 
\begin{equation}\label{Jdetrep} 
p(z_1, z_2) = p(0,0) \det (I-K (z_1I_{n_1}\oplus z_2 I_{n_2})). 
\end{equation}
\end{conjecture} 

In \cite{Jackson} evidence was given for the above conjecture by proving partial results. Inspired by this conjecture, we became interested in matrices $K\in\mathbb{C}^{n \times n}$ and decompositions $n=n_1+\cdots+n_d$ for which the polynomial $p(z)$ given by \eqref{DetRep} has no roots on $\mathbb{T}^d$. This happens when $K$ is strictly $J$-contractive for some diagonal signature matrix $J$, and also if $K$ is {\em structurally similar} to a strict $J$-contraction, that is, when there exists an invertible $S=\oplus_{j=1}^d S_j$ with $S_j\in\mathbb{C}^{n_j \times n_j}$, $j=1,\ldots, d$, such that $ SKS^{-1}$ is a strict $J$-contraction. However, it can happen that $K$ is not structurally similar to a strict $J$-contraction, while $p(z)$ does not have roots on $\mathbb{T}^d$; for $d=2$ Example \ref{ex} below gives a polynomial $p(z)$ with no roots inside $\overline{\mathbb{D}}^2$ and a determinantal representation with a matrix $K$ that is not structurally similar to a strict $J$-contraction, and another determinantal representation where the matrix is a strict $J$-contraction. To distinguish between these two cases, it turns out useful to extend the determinantal representation \eqref{DetRep} to a {\em noncommutative determinantal representation} in which one evaluated over matrices of arbitrary sizes:
\begin{equation}\label{NC-DetRep}
p(U_1, \ldots , U_d) = p(0,\ldots,0) \det (I_n\otimes I_m- (K\otimes I_m) (\oplus_{j=1}^d I_{n_j}\otimes U_j)),
\end{equation}
where $U_1,\ldots , U_d \in {\mathbb C}^{m \times m}$, $m\in {\mathbb N}$. Restricting \eqref{NC-DetRep} to $U_j\in\mathbb{C}^{1 \times 1}$, we recover \eqref{DetRep}. Different choices of $K$ in \eqref{DetRep}, for the same polynomial $p(z)$, can lead to different noncommutative determinantal representations \eqref{NC-DetRep}. However, if two matrices are structurally similar, then their extensions to noncommutative determinantal representations must be the same. As it turns out, the behaviour of determinantal representations in the commutative setting described above, as illustrated in Example \ref{ex} below, cannot occur if one considers their noncommutative extensions, and replaces roots on the $d$-torus with the noncommutative analogue of $d$-tuples of unitary matrices of arbitrary size, as follows from our main result.

\begin{theorem}\label{T:main}
Let $p(z_1,\ldots,z_d)$ be a multivariable polynomial that is given by a determinantal representation
$$ p(z_1, \ldots , z_d) = p(0,\ldots,0) \det (I_n- K (\oplus_{j=1}^d z_j I_{n_j})), $$
with $K \in {\mathbb C}^{n\times n}$ and $n=\sum_{j=1}^d n_j$. Define
$$ p(U_1, \ldots , U_d) = p(0,\ldots,0) \det (I_n\otimes I_m- (K\otimes I_m) (\oplus_{j=1}^d I_{n_j}\otimes U_j)),$$
where $U_1,\ldots , U_d \in {\mathbb C}^{m \times m}$, $m\in {\mathbb N}$.  Then the following are equivalent.
\begin{itemize}
\item[(i)]
$p(U_1, \ldots , U_d)\neq 0$ for all tuples of unitaries $(U_1, \ldots , U_d) \in ({\mathbb C}^{m\times m})^d$, $m\in {\mathbb N}$. 
\item[(ii)] $p(U_1, \ldots , U_d)\neq 0$ for all tuples of unitaries $(U_1, \ldots , U_d) \in ({\mathbb C}^{m\times m})^d$, $m=1,\ldots,n$.
\item[(iii)] There exists an invertible $S=\oplus_{j=1}^d S_j$, with $S_j\in{\mathbb C}^{n_j\times n_j}$, $j=1,\ldots, d$, such that 
$ S^{-1}KS$ is a strict $J$-contraction for some diagonal signature matrix $J$.
\end{itemize}
\end{theorem}

Our main result will be proved in Section \ref{S:DetRepUniSings}. In the course of the proof we will also find a way to determine $d$-tuples of unitaries $(U_1, \ldots , U_d)$ in the case that $K$ is not structurally similar to a strict $J$-contraction, as will be illustrated by examples.   Note that a diagonal signature matrix is simply a diagonal matrix with each diagonal entry equal to either 1 or -1. The above result, could also be stated with $J$ being a signature matrix that is block diagonal with blocks of size $n_j \times n_j$ so that $J$ commutes with $\oplus_{j=1}^d z_j I_{n_j}$, but stating the results for diagonal signature matrices lead to simpler statements that have the same generality.

To conclude this introduction, we recall some standard notation and terminology used throughout this paper. We let ${\mathcal H}_{n}$ denote the vectorspace (over ${\mathbb R}$) of $n\times n$ Hermitian complex matrices. The inner product on ${\mathcal H}_{n}$ is given by $\langle A, B \rangle = {\rm Tr} (AB^*)= {\rm Tr} (AB)$, where ${\rm Tr} $ denotes the trace. We let ${\rm PSD}$ and ${\rm PD}$ denote the cones of positive semidefinite and positive definite matrices, respectively, where the size of the matrices should be clear from the context.

\section{Indefinite determinantal representations versus nonsingularities on the noncommutative $d$-torus}\label{S:DetRepUniSings} 

In this section we shall prove the main result of the paper, Theorem \ref{T:main}, after which we provide some illustrative examples. First we  present some auxiliary results. With a matrix $K \in {\mathbb C}^{n\times n}$ and a decomposition $n=\sum_{j=1}^d n_j$ we associate the following subspace of ${\mathcal H}_n$ 
$$
{\mathcal W}_K= \left\{ (\oplus_{j=1}^d H_{j}) - K (\oplus_{j=1}^d H_{j}) K^* \colon H_j\in {\mathcal H}_{n_j}, j=1,\ldots , d \right\}.
$$ 
We now describe the orthogonal complement ${\mathcal W}_K^\perp$ in $\cH_n$ and the intersection of ${\mathcal W}_K^\perp$ and ${\rm PSD}$.   

\begin{lemma}\label{L:WKperp}
Let $K \in {\mathbb C}^{n\times n}$ and  $n=\sum_{j=1}^d n_j$. The orthogonal complement ${\mathcal W}_K^\perp$ of ${\mathcal W}_K$ in $\cH_n$, with respect to the trace inner product, is given by
\[
{\mathcal W}_K^\perp= \left\{ L\in\cH_n \colon L - K^* L K = \begin{bmatrix}
           0_{n_1} & & {\huge *} \cr  & \ddots & \cr  {\huge *} &  & 0_{n_d}
       \end{bmatrix}  \right\}.
\] 
Moreover, let $L\in\textup{PSD}$ and factor $L$ as $L=XX^*$ for $X\in\mathbb{C}^{n \times m}$. Decompose 
$$X=\mat{X_1\\ \vdots \\ X_d} \mbox{ with } X_j\in\mathbb{C}^{n_j \times m},\ j=1,\ldots,d.$$ 
Then $L\in {\mathcal W}_K^\perp$ if and only if there exist unitary $V_1,\ldots,V_d\in\mathbb{C}^{m \times m}$ such that 
\begin{equation}\label{V1Vd}
K^* \mat{X_1\\ \vdots\\ X_d} =\mat{X_1 V_1\\ \vdots\\ X_d V_d}. 
\end{equation}
\end{lemma}

\begin{proof}
Note that $L\in\cH_n$ is in ${\mathcal W}_K^\perp$ if and only if ${\rm Tr} (LW) \ge 0$ for every $W\in {\mathcal W}_K $. For $W=(\oplus_{j=1}^d H_{j}) - K (\oplus_{j=1}^d H_{j}) K^*\in{\mathcal W}_K$ we have 
\begin{align*}
{\rm Tr} (LW) 
&= {\rm Tr} (L((\oplus_{j=1}^d H_{j}) - K (\oplus_{j=1}^d H_{j}) K^*))\\
&={\rm Tr} (L(\oplus_{j=1}^d H_{j})) - {\rm Tr} (LK (\oplus_{j=1}^d H_{j}) K^*)\\
&={\rm Tr} (L(\oplus_{j=1}^d H_{j})) - {\rm Tr} (K^*LK (\oplus_{j=1}^d H_{j}))
={\rm Tr} ((L- K^*LK)(\oplus_{j=1}^d H_{j})).
\end{align*} 
Hence $L\in{\mathcal W}_K^\perp$ if and only if 
$$ {\rm Tr}( (L-K^*LK) (\oplus_{j=1}^d H_{j}))=0 \mbox{ for every $H_j \in {\mathcal H}_{m_j}$, $j=1,\ldots, d$.}$$ 
The latter is equivalent to $L-K^*LK$ having its diagonal blocks equal to zero, which proves the formula for ${\mathcal W}_K^\perp$.

Let $L\in\textup{PSD}$ be as described in the lemma. Define 
\begin{equation}\label{Y}
Y=\mat{Y_1\\ \vdots\\ Y_d}=K^* \mat{X_1\\ \vdots\\ X_d} \mbox{ with } Y_j\in\mathbb{C}^{n_j \times m},\ j=1,\ldots,d.  
\end{equation}
Then $L\in{\mathcal W}_K^\perp$ is equivalent to $X_iX_i^*=Y_iY_i^*$ for $i=1,\ldots,d$, which in turn is equivalent to  $Y_i=X_i V_i$ for some unitary $V_i\in\mathbb{C}^{m\times m}$, $i=1,\ldots,d$.
\end{proof}

Next we show how $\mathcal{W}_K$ and $\mathcal{W}_K^\perp$ can be used in a criteria to determine if $K$ is structurally similar to a strictly $J$-contractive matrix, for $J$ some diagonal signature matrix. 

\begin{proposition}\label{P:Wequiv}
Let $K \in {\mathbb C}^{n\times n}$ and decompose $n=\sum_{j=1}^d n_j$. Then the following are equivalent.
   \begin{itemize}
       \item[(i)] There exists an invertible $S=\oplus_{j=1}^d S_j$, with $S_j \in {\mathbb C}^{n_j\times n_j}$, $j=1,\ldots , d$,  so that $S^{-1}KS$ is $J$-contractive for a diagonal signature matrix $J$.
       \item[(ii)] We have ${\mathcal W}_K \cap {\rm PD} \neq \emptyset$.
       \item[(iii)] We have ${\mathcal W}_K^\perp \cap {\rm PSD} =\{ 0 \}$.
   \end{itemize}
    \end{proposition}

\begin{proof}
(i) $\to$ (ii). Let $S=\oplus_{j=1}^d S_j$ be so that $S^{-1}KS$ is strictly $J$-contractive for the diagonal signature matrix $J=\oplus_{j=1}^d J_j$, with $J_j\in {\mathcal H}_{n_j}$, that is, 
\begin{equation}\label{JKJK} J - S^{-1}KS J S^{*} K S^{*-1} >0. \end{equation} 
Define $H = SJS^*=\oplus_{j=1}^d S_{j}J_jS_j^*=:\oplus_{j=1}^d H_{j}$. Multiplying in \eqref{JKJK} with $S$ on the left and with $S^*$ in the right, 
we find that $H-KHK^*\in {\mathcal W}_K\cap {\rm PD}$. 
    
(ii) $\to$ (i). Let $(\oplus_{j=1}^d H_{j}) - K (\oplus_{j=1}^d H_{j}) K^* \in {\mathcal W}_K\cap {\rm PD}$. By \cite[Section 13.2, Theorem 2]{LancasterTismenetsky}, we have that $\oplus_{j=1}^d H_{j}$ is invertible. We may now write $H_j = S_jJ_jS_j^*$, $j=1,\ldots , d$, where $S_1, \ldots , S_d$  are invertible and $J_1, \ldots ,J_d$ are diagonal real unitaries (e.g., one can take a spectral decomposition $H_j=U_jD_jU_j^*$, and let $J_j=D_j|D_j|^{-1}$ and $S_j=U_j|D_j|^\frac12$). For this choice of $S$ we have \eqref{JKJK}, and hence (i) holds. 

The equivalence of (ii) and (iii) follows from a Hahn-Banach separation hyperplane argument. Indeed, ${\mathcal W}_K \cap {\rm PD} = \emptyset $ if and only if we can find a linear functional in the Hilbert space of Hermitian matrices separating the convex sets ${\mathcal W}_K $ and $ {\rm PD} $. In other words, we can find a $L\in {\mathcal H}_{n}$, so that ${\rm Tr} (LP)>0$ for every $P \in  {\rm PD}$, and ${\rm Tr} (LW) \le 0$ for every $W\in {\mathcal W}_K$. Since ${\mathcal W}_K$ is a vector space (over $\mathbb{R}$), this implies that ${\rm Tr} (LW) = 0$ for every $W\in {\mathcal W}_K$, for if ${\rm Tr} (LW) < 0$ would hold for a $W\in {\mathcal W}_K$, then ${\rm Tr} (L(-W)) > 0$. Hence $L\in {\mathcal W}_K^\perp$. And ${\rm Tr} (LP)>0$ for every $P \in  {\rm PD}$ is equivalent to $0\neq L \in {\rm PSD}$.
\end{proof}

Combining Proposition \ref{P:Wequiv} and Lemma \ref{L:WKperp} we can connect the elements of $\mathcal{W}_K^\perp \cap \textup{PSD}$ and points $(U_1,\ldots,U_d)$ on the noncommutative $d$-torus for which $p(U_1,\ldots,U_d)=0$. 

\begin{lemma}\label{L:pzeroes}
Assume $\mathcal{W}_K^\perp \cap \textup{PSD}\neq \{0\}$. Let $0\neq L\in \mathcal{W}_K^\perp \cap \textup{PSD}$ factor as $L=XX^*$ with $X\in \mathbb{C}^{n \times m}$ and let $V_1,\ldots,V_d$ be unitary matrices in $\mathbb{C}^{m \times m}$ such that \eqref{V1Vd} holds. Then $p(U_1,\ldots,U_d)=0$ for $U_j=V_j^T$, $j=1,\ldots,d$, more concretely,  ${\rm vec}(X^*)$ is in the cokernel of $I_n \otimes I_m- (K\otimes I_m) (\oplus_{j=1}^d I_{n_j}\otimes V_j^T)$. Moreover, all unitary $d$-tuples $(U_1,\ldots,U_d)$ such that $p(U_1,\ldots,U_d)=0$ are obtained in that way.
\end{lemma}  

\begin{proof}
Let $0\neq L\in \mathcal{W}_K^\perp \cap \textup{PSD}$ and let $X$ and $V_1,\ldots,V_d$ be as in the lemma. Taking adjoints on both sides in \eqref{V1Vd} and applying the vectorization operation, on the left hand side we get ${\rm vec} (X^* K)=K^T\otimes I_m {\rm vec}(X^*)$ while on the right hand side we have 
\begin{align*}
{\rm vec}(\mat{V_1^*X_1^*&\ldots&V_d^*X_d^*})
&=\mat{{\rm vec}(V_1^*X_1^*)\\ \vdots\\ {\rm vec}(V_d^*X_d^*)}
=\mat{(I_{n_1}\otimes V_1^*){\rm vec}(X_1^*)\\ \vdots\\ (I_{n_d}\otimes V_d^*){\rm vec}(X_d^*)}\\
&=(\oplus_{j=1}^dI_{n_j}\otimes V_j^*) {\rm vec}(X^*).
\end{align*}  
Multiplying with the unitary matrix $\oplus_{j=1}^dI_{n_j}\otimes V_j$ on both sides and taking transposes it follows that
\[
{\rm vec}(X^*)^T(I_{n}\otimes I_m - (K\otimes I_m)(\oplus_{j=1}^dI_{n_j}\otimes V_j^T))=0,
\]
which shows that $p(V_1^T,\ldots,V_d^T)=0$. For the converse implication, assume that $U_1,\ldots,U_d$ are $m \times m$ unitary matrices such that $p(U_1,\ldots,U_d)=0$, set $V_j=U_j^T$, $j=1,\ldots,d$, and let $v\in \mathbb{C}^{nm}$ be such that $v^T (I_{n}\otimes I_m - (K\otimes I_m)(\oplus_{j=1}^dI_{n_j}\otimes V_j^T))=0$. Let $X\in\mathbb{C}^{n \times m}$ be the matrix such that ${\rm vec}(X^*)=v$. Reversing the above computation shows that \eqref{V1Vd} holds, and hence $L=XX^*$ is in $\mathcal{W}_K^\perp \cap \textup{PSD}$ by Lemma \ref{L:WKperp}.
\end{proof}

We now prove our main result.

\begin{proof}[Proof of Theorem \ref{T:main}]
(iii) $\to$ (i). If (iii) holds, then $\mathcal{W}_K^\perp \cap \textup{PSD}=\{0\}$, by Proposition \ref{P:Wequiv}, which in turn implies that $p(U_1,\ldots,U_d)\neq 0$ for any $d$-tuple of unitary matrices $(U_1,\ldots,U_d)$, by Lemma \ref{L:pzeroes}.

(i) $\to$ (ii). Trivial.

(ii) $\to$ (iii). We show that (ii) implies that ${\mathcal W}_K^\perp \cap {\rm PSD}= \{ 0 \}$, so that (iii) follows by Proposition \ref{P:Wequiv}. Assume to the contrary that there exists an $0\neq L\in {\mathcal W}_K^\perp \cap {\rm PSD}$ and factor $L=XX^*$ with $X\in\mathbb{C}^{n \times m}$ with $m={\rm rank}(L)\leq n$. By Lemma \ref{L:pzeroes} we get that $p(U_1,\ldots,U_d)=0$ for a $d$-tuple of unitaries $(U_1,\ldots,U_d)$ of size $m \times m$, in contradiction with (ii). Hence (ii) implies (iii). 
\end{proof}

Since the zeroes of $p(U_1,\ldots,U_d)$ on the noncommutative $d$-torus are connected to elements of the cone ${\mathcal W}_K^\perp \cap {\rm PSD}$, it might be of interest to understand the structure of this cone better. In the next proposition we determine a necessary and sufficient condition under which an element in ${\mathcal W}_K^\perp \cap {\rm PSD}$ is an extreme ray. For a matrix $Z=(z_{kl})_{k,l}$ we define $\overline{Z}=(\overline{z_{kl}})_{k,l}=(Z^*)^T$. 

\begin{proposition} 
Let $K \in {\mathbb C}^{n\times n}$ and decompose $n=\sum_{j=1}^d n_j$. Factor $0\neq L\in {\mathcal W}_K^\perp \cap {\rm PSD}$ as $L=XX^*$ with $X\in\mathbb{C}^{n \times m}$, with $m={\rm rank}(L)$, and define $Y$ by \eqref{Y}. Then $L$ generates an extreme ray if and only if 
$${\rm dim}\cap_{j=1}^d\Ker (\overline{X}_j \otimes X_j - \overline{Y}_j \otimes Y_j)=1.$$ 
\end{proposition}

\begin{proof}
Let $0\neq L\in {\mathcal W}_K^\perp \cap {\rm PSD}$ and write $L=XX^*$ as in Lemma \eqref{L:WKperp} with $m={\rm rank}L$ and define $Y$ as in \eqref{Y}. Then $L$ generates an extreme ray if and only if whenever  $L=L_1+L_2$ with $L_1,L_2\in{\mathcal W}_K^\perp \cap {\rm PSD}$ we necessarily have that $L_1,L_2\in {\rm span}\{L\}$. Assume that $L=L_1+L_2$ with $L_1$ and $L_2$ as above. Since for $k=1,2$ we have $0\leq L_k\leq L$ it follows that $L_k=XH_kX^*$ with $0\leq H_k\leq I$ and $H_1+H_2=I_m$ and $L_k\in {\rm span}\{L\}$ corresponds to $H_k\in {\rm span}\{I_m\}$, while $L_k\in{\mathcal W}_K^\perp$ corresponds to $X_jH_kX_j-Y_jH_kY_j=0$ for $j=1,\ldots,d$. Since $\Ker X=\{0\}$ (because $m={\rm rank}(L)$), $L_k$ and $H_k$ determine each other uniquely, and it follows that $L$ generates an extreme ray if and only if for all  $H\in {\rm PSD}$ with $H\leq I$ (take $H_1=H$ and $H_2=I-H$) and $X_jHX_j^*-Y_jHY_j^*=0$ for $j=1,\ldots,d$ we must have that $H\in {\rm span}\{I_m\}$.  

Note that for any $H\in\mathbb{C}^{m \times m}$ and for $j=1,\ldots,d$ we have
\[
{\rm vec} (X_jHX_j^*-Y_jHY_j^*)= (\overline{X}_j \otimes X_j - \overline{Y}_j \otimes Y_j) {\rm vec}(H).
\]
Denote $R_j:=\overline{X}_j \otimes X_j - \overline{Y}_j \otimes Y_j$, $j=1,\ldots,d$. Since $X_jX_j^*=Y_jY_j^*$, it follows that ${\rm vec}(I_m)$ is in $\Ker R_j$ for each $j$. When $L$ does not generate an extreme ray, then there exists an $0\leq H\leq I$ such that $H\not\in{\rm span}\{I_m\}$ and ${\rm vec}(H)\in \Ker R_j$ for each $j$, so that $\cap_{j=1}^d \Ker R_j$ has dimension at least 2. Conversely, assume that the dimension of $\cap_{j=1}^d \Ker R_j$ is at least 2. Then there exists a $K\not\in{\rm span}\{I_m\}$, such that ${\rm vec}(K)\in \Ker R_j$ for each $j$, and thus $X_jKX_j^*-Y_jKY_j^*=0$ for $j=1,\ldots,d$, however, $K$ need not be in ${\rm PSD}$. Write $K=K_1+ i K_2$ with $K_1,K_2\in \cH_m$. Then also $X_jK_kX_j^*-Y_jK_kY_j^*=0$ for $j=1,\ldots,d$ and $k=1,2$ and $K_1$ and $K_2$ cannot both be in ${\rm span}\{I_m\}$. Hence we may assume without loss of generality that $K\in \cH_m$. Since $\cap_{j=1}^d \Ker R_j$ is a vector space containing both $K$ and $I_m$, it follows that 
$$H=\|K+\|K\|I_m\|^{-1}(K+\|K\|I_m)\in \cap_{j=1}^d \Ker R_j$$ 
satisfies $0\leq H \leq I$ and $H\notin{\rm span}\{I_m\}$. Hence $L$ does not generate an extreme ray.
\end{proof}

We next provide a sufficient condition under which an $0\neq L\in {\mathcal W}_K^\perp \cap {\rm PSD}$ does not generate an extreme ray. 

\begin{lemma}
Factor $0\neq L\in {\mathcal W}_K^\perp \cap {\rm PSD}$ as $L=XX^*$ with $X\in\mathbb{C}^{n \times m}$, where $m={\rm rank}(L)$, and let $V_1,\ldots,V_d$ be $m \times m$ unitary matrices as in \eqref{V1Vd}. If $V_1,\ldots,V_d$ have a common nontrivial invariant subspace, then $L$ does not generate an extreme ray of ${\mathcal W}_K^\perp \cap {\rm PSD}$. 
\end{lemma}

\begin{proof} Let ${\mathcal{C}}$ be a common nontrivial invariant subspace of $V_1,\ldots,V_d$, and let  
$P_{\mathcal{C}}$ be the orthogonal projection onto $\mathcal{C}$. Now we may write
$$ L= XP_{\mathcal{C}}X^* + X(I- P_{\mathcal{C}}) X^* =: L_1 + L_2. $$
We claim that $L_1, L_2 \in {\mathcal W}_K^\perp \cap {\rm PSD}$. Clearly, $L_1, L_2 \in {\rm PSD}$. Next, 
$$ L_1 - K^* L_1 K = \mat{X_1\\ \vdots\\ X_d} P_{\mathcal{C}} \mat{X_1^* & \dots &  X_d^* } - \mat{X_1V_1\\ \vdots\\ X_dV_d} P_{\mathcal{C}} \mat{V_1^*X_1^* & \dots &  V_d^* X_d^* }. $$
Since $V_iP_{\mathcal{C}}V_i^* = P_{\mathcal{C}}$, $i=1,\ldots, d,$ we obtain that the block diagonal entries of $L_1 - K^* L_1 K$ are zero. Thus $L_1\in {\mathcal W}_K^\perp$. Then, also $L_2=L-L_1 \in {\mathcal W}_K^\perp$, which finishes the proof.
\end{proof}

The following example shows that the condition above is indeed only sufficient.

\begin{example}  \rm
Let 
$$ K= \begin{bmatrix}
    1 & 0 & 0 & \frac{1}{\sqrt{2}} \cr 
    0 & -1 & 0 & 0 \cr 
    0 & 0 & i & \frac{1}{\sqrt{2}} \cr 
    0 & 0 & 0 & 0 
\end{bmatrix} \quad\mbox{and}\quad L= \begin{bmatrix}
    1 & 0 & 0 & 0 \cr 
    0 & 1 & 0 & 1 \cr 
    0 & 0 & 1 & 0 \cr 
    0 & 1 & 0 & 1 
\end{bmatrix}. 
$$
In this case we find that $$
X_1=I_3, X_2=\begin{bmatrix} 0 & 1 & 0 \end{bmatrix}, Y_1= \begin{bmatrix} 1 & 0 & 0 \cr 0 & -1 & 0 \cr 0 & 0 & -i  \end{bmatrix},
Y_2=\begin{bmatrix} \frac{1}{\sqrt{2}} & 0 & \frac{1}{\sqrt{2}}\end{bmatrix}.$$ This forces
$$ V_1 = \begin{bmatrix} 1 & 0 & 0 \cr 0 & -1 & 0 \cr 0 & 0 & -i  \end{bmatrix} , V_2 = \begin{bmatrix} \frac{\beta}{\sqrt{2}} \sin t & \gamma \cos t &  \frac{-\beta}{\sqrt{2}} \sin t\cr \frac{1}{\sqrt{2}} & 0 & \frac{1}{\sqrt{2}} \cr \frac{\alpha\beta}{\gamma\sqrt{2}} \cos t & -\alpha \sin t & \frac{-\alpha\beta}{\gamma\sqrt{2}} \cos t  \end{bmatrix}, $$
where $|\alpha|=|\beta|=|\gamma|=1$ and $t\in {\mathbb R}$. The only invariant subspaces of $V_1$ are spanned by standard basis vectors, but any such nontrivial subspace is not an invariant subspaces for $V_2$. In addition, $L$ does not generate an extreme ray of ${\mathcal W}_K^\perp \cap {\rm PSD}$, as one may write $L$ as 
$$ L = \begin{bmatrix} X_1 \cr X_2 \end{bmatrix} \begin{bmatrix} 1 & 0 & 0 \cr 0 & \frac12 & 0 \cr 0 & 0 & 0  \end{bmatrix} \begin{bmatrix} X_1^* & X_2^* \end{bmatrix} +  \begin{bmatrix} X_1 \cr X_2 \end{bmatrix} \begin{bmatrix} 0 & 0 & 0 \cr 0 & \frac12 & 0 \cr 0 & 0 & 1  \end{bmatrix} \begin{bmatrix} X_1^* & X_2^* \end{bmatrix}=: L_1 + L_2,$$
with $L_1, L_2 \in {\mathcal W}_K^\perp \cap {\rm PSD}$.
This shows that it can happen that the unitaries $V_i$, $i=1,2$, do not have a common nontrivial invariant subspace, while $L$ still does not generate an extreme ray of ${\mathcal W}_K^\perp \cap {\rm PSD}$.
\end{example}

\begin{remark}\rm 
    Note that in Theorem \ref{T:main}(ii) it suffices to consider $m$ ranging from $1$ to $\max\{1,n-1\}$, as in the proof of (ii) $\to$ (iii) it suffices to consider matrices $L$ that generate an extreme ray. When $n\ge 2$ any extreme ray will consist of singular matrices, and thus they have rank at most $n-1$. 
\end{remark}

The following example illustrates that the same polynomial will have different determinantal representations, some of which can be made into a $J$-contractive determinantal representation via a block diagonal similarity, while other determinantal representations do not allow for this. 

    \begin{example}\label{ex}\rm
        Let 
        
    $$ K = \begin{bmatrix}
            \frac12 & 0 & 0 & \frac12 \cr 
            0 & -\frac12 & \frac12 & 0 \cr \frac12
            & 0 & 0 & \frac12 \cr 0& -\frac12 & \frac12 & 0  
        \end{bmatrix} $$ and $n_1=n_2=2$. Then the corresponding polynomial $p$ equals 
        $$ p(z_1,z_2)=1-\frac14 (z_1^2+z_2^2).$$ If we let 
        $$ L= \begin{bmatrix}
            1 & 0 & 1 & 0 \cr 
            0 & 1 & 0 & 1 \cr 
            1 & 0 & 1 & 0 \cr 
            0 & 1 & 0 & 1 
        \end{bmatrix} = \mat{I_2\\I_2} \mat{I_2 & I_2}=XX^*.$$ Then 
        \begin{align*}
K^* X=\mat{1&0\\0&-1\\0&1\\1&0}=\mat{I_2 V_1\\ I_2 V_2}\mbox{ with }V_1=\mat{1&0\\0&-1},\, V_2=\mat{0&1\\1&0}.
        \end{align*}
Hence $L \in {\mathcal W}_K^\perp \cap {\rm PSD} $, by Lemma \ref{L:WKperp}. Thus by Proposition \ref{P:Wequiv} we have that $K$ cannot be made into a strict $J$-contraction by a block diagonal similarity, and it follows by Lemma \ref{L:pzeroes} that $p(U_1,U_2)=0$ for $U_i=V_i^T=V_i$, $i=1,2$, as can easily be computed.

Of course, as $p(z)$ is strictly stable, a determinantal representation exists for it with (a different) $K$ that is a strict contraction; see \cite{Kummert, GKVW2013}.  Indeed, if we apply the procedure in \cite{GKVW2013}, we find that $p(z_1,z_2)= \det(I_4 - \widetilde{K}(z_1I_2\oplus z_2I_2))$ with 
$$\widetilde{K}=\begin{bmatrix}
    \frac12 & 0& \frac14& 0\cr0&-\frac12&-\frac14&0\cr0&0&0&1\cr\frac14&\frac14&\frac14&0\end{bmatrix}.
$$ After a similarity with $S={\rm diag}(1,1,1,\frac34)$, we find that $S^{-1}\widetilde{K}S$ is a strict contraction (thus $J=I$ in this case). One also easily checks that with $U_1$ and $U_2$ given above we have $\det(I_4\otimes I_2 - \widetilde{K}\otimes I_2(I_2\otimes U_1\oplus I_2\otimes U_2))=1/4$.
\end{example}

\subsection*{Acknowledgements}
This work is based on research supported in part by the National Research Foundation of South Africa (NRF, Grant Numbers 118513 and 127364) and the DSI-NRF Centre of Excellence in Mathematical and Statistical Sciences (CoE-MaSS). Any opinion, finding and conclusion or recommendation expressed in this material is that of the authors and the NRF and CoE-MaSS do not accept any liability in this regard. We also gratefully acknowledge funding from the National Graduate Academy for Mathematical and Statistical Sciences (NGA(MaSS)) and the University Capacity Development Programme (UCDP) from the Department of Higher Education and Training (DHET). In addition, HW is partially supported by United States National Science Foundation (NSF) grant DMS-2000037.

\bibliographystyle{plain}

\end{document}